\documentclass{amsart}
\usepackage{amsmath,amsthm,verbatim,amssymb,amsfonts,amscd,graphicx}
\usepackage{xcolor}
\usepackage{graphics}

\usepackage{euscript, enumerate}
\usepackage{url,hyperref}

\newcommand{\e}{\epsilon}

\newcommand{\eps}{\varepsilon}

\newcommand{\be}{\begin{equation}}
\newcommand{\ba}{\begin{aligned}}
\newcommand{\bee}{\begin{equation*}}
\newcommand{\ee}{\end{equation}}
\newcommand{\ea}{\end{aligned}}
\newcommand{\eee}{\end{equation*}}
\newcommand{\bea}{\begin{equation} \begin{aligned} }
\newcommand{\eea}{\end{aligned}\end{equation} }

\theoremstyle{plain}
\newtheorem{theorem}{Theorem}[section]

\newtheorem{prop}[theorem]{Proposition}

\theoremstyle{remark}

\theoremstyle{definition}

\numberwithin{equation}{section}

\begin{document}
\title{A note on blowup limits in 3d Ricci flow}
\author{Beomjun Choi, Robert Haslhofer}

\begin{abstract}
We prove that Perelman's ancient ovals occur as blowup limit in 3d Ricci flow through singularities if and only if there is an accumulation of spherical singularities.
\end{abstract}
\maketitle

\section{Introduction}

In \cite[Section 1.4]{Per2}, Perelman constructed ancient ovals for the Ricci flow. These are ancient 3-dimensional Ricci flows that for $t\to 0$ converge to a round point, but for $t\to -\infty$ look like a long neck capped off by two Bryant solitons. Uniqueness of Perelman's ancient ovals has been proved in recent important work by Brendle-Daskalopoulos-Sesum \cite{BDS}. While 3-dimensional Ricci flow is by now extremely well understood, one of the few remaining open problems is whether or not Perelman's ancient ovals actually occur as blowup limit in 3d Ricci flow through singularities. In this short note, we prove that Perelman's ancient ovals occur as blowup limit  if and only if there is an accumulation of spherical singularities.\\

For the purpose of this note it is most convenient to describe 3-dimensional Ricci flow through singularities as metric flows, as introduced by Bamler \cite{Bam2}.\footnote{{In particular, as sketched in \cite[Section 3.7]{Bam2}, after selecting a branch and passing to the metric completion, any 3-dimensional singular Ricci flow from Kleiner-Lott \cite{KL} can be viewed as a metric flow.}}
{A metric flow
is given by a set $\mathcal{X}$,} a time-function $\mathfrak{t}:\mathcal{X}\to \mathbb{R}$, complete separable metrics $d_t$ on the time-slices $\mathcal{X}_t=\mathfrak{t}^{-1}(t)$, and probability measures $\nu_{x;s}\in \mathcal{P}(\mathcal{X}_s)$ such that the Kolmogorov consistency condition and a certain sharp gradient estimate for the heat flow hold (see Section \ref{sec_prelim} for details).
We remark that by our recent work \cite{CH}, any metric flow $\mathcal{X}$ {arising as a noncollapsed limit of smooth Ricci flows or arising from a 3-dimensional singular Ricci flow from \cite{KL}} is a weak solution of the Ricci flow in the sense of Haslhofer-Naber \cite{HN}, namely for almost every $(p,t)$ the estimate
$|\nabla_p\mathbb{E}_{(p,t)}[F]| \leq \mathbb{E}_{(p,t)}[|\nabla^\parallel F|]$
holds for all cylinder functions $F$ on path-space. However, for our present purpose it is enough to assume that the Ricci flow equation
$
 \mathcal{L}_{\partial_{\mathfrak{t}}} g = -2 \textrm{Ric}(g)
$ holds on the regular part.\\
 
Next, we recall from \cite[Section 6.8]{Bam2} that a tangent flow at some space-time point $x\in \mathcal{X}$ is an $\mathbb{F}$-limit of parabolic rescalings by any sequence $\lambda_i\to\infty$ of $(\mathcal{X},(\nu_{x;s})_{s\leq \mathfrak{t}(x)})$, where as above $\nu_{x;s}$ denotes the conjugate heat kernel measures centered at $x$. We say that $\mathcal{X}$ has a spherical singularity at $x$, if some (and thus any) tangent flow at $x$ is a round shrinking sphere. More generally than tangent flows, which are given as rescaling limits around a fixed center $x$, one can consider rescalings around any sequence of space-time points $x_i\to x$. A blowup limit at $x$ is any $\mathbb{F}$-limit of parabolic rescalings by any sequence $\lambda_i\to\infty$ of $(\mathcal{X},(\nu_{x_i;s})_{s\leq \mathfrak{t}(x_i)})$. While tangent flows are always self-similarly shrinking as a consequence of Perelman's monotonicity formula, the analysis of general blowup limits is much more involved. Finally, we say that there is an accumulation of spherical singularities at $x\in\mathcal{X}$, if the flow $\mathcal{X}$ has spherical singularities at a sequence of pairwise distinct space-time points $x_i\to x$.
Using the above notions, we can now state our main result:

 \begin{theorem}\label{thm_equivalence} 
Perelman's ancient ovals occur as blowup limit in a 3-dimensional Ricci flow through singularities if and only if there is an accumulation of spherical singularities.\end{theorem}

In essence, we use a similar scheme of proof as in our prior joint work with Hershkovits \cite{CHH}, where we proved a corresponding result for the mean curvature flow. Namely, if there is a sequence of spherical singularities converging to a (quotient) neck singularity, then we blow up by the {largest} spherical scale to construct a nonselfsimilar blowup limit, and conversely we show by contradiction that Perelman's ancient ovals cannot occur as blowup limit if there are only finitely many spherical singularities. However, while the theory of mean curvature flow through singularities has been very well developed over the past 40 years, the study of Ricci flow through singularities was only initiated recently \cite{KL,HN,Sturm,BK,Bam2,Bam3,CH}. Hence, we need to be more careful to set up the geometric analytic framework, and several properties that are well known or obvious for mean curvature flow -- such as quantitative differentiation and change of base-points -- have to be implemented with more care. Another new feature in the setting of Ricci flow are quotients by finite isometry groups. Hence, when analyzing the almost equality case of the monotonicity formula we have to include the possibility of nontrivial spherical space forms and the $\mathbb{Z}_2$-quotients of the cylinder.\\

Finally, motivated by the recent proof of the mean-convex neighborhood conjecture \cite{CHH_mean,CHHW} and the higher-dimensional uniqueness result from Brendle-Daskalopoulos-Naff-Sesum \cite{BDNS}, it seems likely that there is a version of Theorem \ref{thm_equivalence} for Ricci flow through neck-singularities in higher dimensions. However, let us also remark that while blowup limits in 3d Ricci flow are always modelled on $\kappa$-solutions, quotient necks  in higher dimensions can lead to new phenomena. In particular, examples of 4d steady solitons asymptotic to {quotient} necks are quotients of the 4d-Bryant soliton, which have an orbifold singularity, and Appleton's solitons \cite{App}, which are asymptotic to $\mathbb{R}\times S^3/\mathbb{Z}_k${, $k\ge 3$,} and have curvatures of mixed signs.\\

\noindent\textbf{Acknowledgments.} The second author has been partially supported by an NSERC Discovery Grant (RGPIN-2016-04331) and a Sloan Research Fellowship. {We thank the referees for their detailed comments and suggestions.}\\

\section{Notation and preliminaries}\label{sec_prelim}

As introduced by Bamler \cite[Definition 3.2]{Bam2} a metric flow over $I\subseteq\mathbb{R}$, 
\begin{equation}
\mathcal{X}=\left(\mathcal{X},\mathfrak{t},(d_t)_{t\in I},(\nu_{x;s})_{x\in \mathcal{X}, s\in I,s\leq \mathfrak{t}(x)}\right),
\end{equation}
consists of a set $\mathcal{X}$, a time-function $\mathfrak{t}:\mathcal{X}\to \mathbb{R}$, complete separable metrics $d_t$ on the time-slices $\mathcal{X}_t=\mathfrak{t}^{-1}(t)$, and probability measures $\nu_{x;s}\in \mathcal{P}(\mathcal{X}_s)$, called conjugate heat kernel measures, such that:
\begin{itemize}
\item $\nu_{x;\mathfrak{t}(x)}=\delta_x$ for all $x\in \mathcal{X}$, and for all $t_1\leq t_2\leq t_3$ in $I$ and all $x\in\mathcal{X}_{t_3}$ we have the Kolmogorov consistency condition $
\nu_{x; t_1} = \int_{\mathcal{X}_{t_2}} \nu_{\cdot; t_1}\, d\nu_{x; t_2}$.
\item For all $s<t$ in $I$, any $T>0$, and any $T^{-1/2}$-Lipschitz function $f_s:\mathcal{X}_s\to\mathbb{R}$, setting $v_s=\Phi\circ f_s$, where $\Phi:\mathbb{R}\to (0,1)$ denotes the antiderivative of $(4\pi)^{-1}e^{-y^2/4}$, the function
$
v_t:\mathcal{X}_t\to \mathbb{R},\, x \mapsto \int_{\mathcal{X}_s} v_s \, d\nu_{x; s}
$
is of the form $v_t=\Phi\circ f_t$ for some $(t-s+T)^{-1/2}$-Lipschitz function $f_t:\mathcal{X}_t\to\mathbb{R}$.
\end{itemize}
In particular, on any metric flow we always have a heat flow of integrable functions and a conjugate heat flow of probability measures, which are defined for $s\leq \mathfrak{t}(x)$ via the formulas
\begin{equation}\label{heat_flow_def}
v_{\mathfrak{t}(x)}(x):= \int_{\mathcal{X}_s} v_s \, d\nu_{x; s},\qquad
\mu_s:= \int_{\mathcal{X}_t} \nu_{x; s} \, d\mu_{\mathfrak{t}(x)}(x)\, .
\end{equation}
We recall from \cite[Definition 3.30 and Definition 4.25]{Bam2} that a metric flow $\mathcal{X}$ is called $H_n$-concentrated, where $H_n=(n-1)\pi^2/2+4$,  if for all $s\leq t$ in $I$ and all $x_1,x_2\in \mathcal{X}_t$ we have the variance bound
\begin{equation}
\textrm{Var}(\nu_{x_1; s}, \nu_{x_2; s})\leq d_t^2(x_1,x_2)+H_n(t-s),
\end{equation}
and is called future-continuous at $t_0\in I$ if for all conjugate heat flows $(\mu_t)_{t\in I'}$ with finite variance and $t_0\in I'$, the function
$
t\mapsto \int_{\mathcal{X}_t}\int_{\mathcal{X}_t} d_t\, d\mu_t \, d\mu_t
$
is right continuous at $t_0$. {Throughout this note, we will work with an $H_3$-concentrated future-continuous metric flow $\mathcal{X}$ that satisfy the partial regularity results from \cite{Bam3}, and such that $\mathcal{L}_{\partial_{\mathfrak{t}}} g = -2 \textrm{Ric}(g)$ and $R(g)\geq -6$ holds on the regular part.}

Let $\mathcal{X}$ be a metric flow as above. Given any rescaling factors $\lambda_i\to \infty$ and space-time points $x_i\to x$, we consider the sequence of parabolically rescaled flows
\begin{equation}
\mathcal{X}^{x_i,\lambda_i}=\left(\mathcal{X},\lambda_i^2(\mathfrak{t}-\mathfrak{t}(x_i)),(d_{\lambda_i^2(t-\mathfrak{t}(x_i))})_{t\in I},(\nu_{x;\lambda_i^2(s-\mathfrak{t}(x_i))})_{x\in \mathcal{X}, s\in I,s\leq \mathfrak{t}(x)}\right),
\end{equation}
equipped with the parabolically rescaled adjoint heat kernel measures
\begin{equation}
\nu_{s}^{x_i,\lambda_i}=\nu_{x_i, \mathfrak{t}(x_i)+\lambda_i^{-2}s}.
\end{equation}
By Bamler's compactness theory \cite{Bam2}, the metric flow pair $(\mathcal{X}^{x_i,\lambda_i},(\nu_{s}^{x_i,\lambda_i})_{s\leq 0})$ subsequentially converges to a limit $(\mathcal{X}^\infty,(\nu^\infty_{x_{\max};s})_{s\leq 0})$, called a blowup limit at $x$. A priori the limit is just a metric flow pair  obtained in the sense of $\mathbb{F}$-convergence on compact time-intervals within some correspondence, as defined in \cite[Section 6]{Bam2}. However, since we are working in dimension 3, by \cite[Theorem {2.44}]{Bam3} (see also Perelman \cite{Per1,Per2}) all blowup limits  are $\kappa$-noncollapsed, have nonnegative curvature, and are smooth with bounded curvature on compact time intervals. Hence, by the local regularity theorem \cite[Theorem {2.29}]{Bam3} (see also Hein-Naber \cite{HeinN}) the convergence is actually locally smooth. {Finally, recall that if $(\mathcal{X}^{x_i,\lambda_i},(\nu_{s}^{x_i,\lambda_i})_{s\leq 0})$ converges to a blowup limit $(\mathcal{X}^\infty,\nu^\infty)$, and if for some sequence of space-time points $\widetilde{x}_i$  the sequence of probability measures $\widetilde{\nu}^i=(\nu_{s}^{\widetilde{x}_i,{\lambda}_i})_{s\leq 0}$ converges to the same limiting measure $\nu^\infty$, then by Bamler's change of base-point theorem \cite[Theorem 6.40]{Bam2} the sequence $(\mathcal{X}^{\widetilde{x}_i,\lambda_i},\widetilde{\nu}^i)$ also converges to the same limit $(\mathcal{X}^\infty,\nu^\infty)$.}

\section{The proofs}

To prove Theorem \ref{thm_equivalence}, we proceed by establishing the following two propositions.

\begin{prop}\label{prop1}
If there is a sequence of pairwise distinct spherical singularities $x_i$ converging to a neck or quotient neck singularity at $x\in\mathcal{X}$, then Perelman's ancient ovals occur as blowup limit at $x$.
\end{prop}

Here, by neck or quotient neck singularity at $x$ we mean that some tangent flow at $x$ is either a round shrinking $S^2\times\mathbb{R}$ or one of its $\mathbb{Z}_2$-quotients, respectively.

\begin{proof}
Fix a small enough constant $\eps>0$. Given $x\in\mathcal{X}$, we consider the rescaled and restricted flow
\begin{equation}
\mathcal{X}^{\alpha}_x:=\left(\mathcal{X}^{x,{1}/{r_{\alpha}}}|_{(-\eps^{-1},0]}, (\nu ^{x, 1/{r_\alpha}}_s)_{s\in (-\eps^{-1},0]}  \right)
\end{equation}
on dyadic scales $r_\alpha=2^\alpha$, where $\alpha\in\mathbb{Z}$. We say that $\mathcal{X}$ is $\eps$-selfsimilar around $x$ at scale $r_\alpha$, if
\begin{equation}
d_{\mathbb{F}}(\mathcal{X}^{\alpha}_x,\mathcal{S}) <\eps
\end{equation}
for some metric soliton $\mathcal{S}$ that becomes extinct at time zero, where $d_{\mathbb{F}}$ denotes the $\mathbb{F}$-distance on the time interval $(-\eps^{-1},0]$ between metric flow pairs \cite[Definition 5.8]{Bam2}.
Since we are working in dimension 3, as a consequence of \cite[Theorem {2.44}]{Bam3} and \cite{Per2} the only metric solitons are flat $\mathbb{R}^3$, the round shrinking $S^3$, the round shrinking $\mathbb{R}\times S^2$, as well as finite quotients thereof (note also that a lower bound for the Nash entropy, which we always have near any neck singularity, gives an upper bound for the order of the quotient group). In particular, as a consequence of the local regularity theorem \cite[Theorem {2.29}]{Bam3} we actually have
\begin{equation}
d_{C^{\lfloor 1/ \tilde{\eps}\rfloor}}(\mathcal{X}^{\alpha}_x,\mathcal{S}) <\tilde{\eps}
\end{equation}
with $\tilde{\eps}(\eps)\to 0$ as $\eps\to 0$, where  the $C^{\lfloor 1/ \tilde{\eps}\rfloor}$-distance is modulo diffeomorphisms on $B_{1/\tilde{\eps}}\times (-\tilde{\eps}^{-1},\tilde{\eps}]$.\\

For any space-time point $x\in\mathcal{X}$ we denote by $S(x)$ the largest spherical scale, i.e. the supremum of $r_\alpha$ such that $\mathcal{X}$ is $\eps$-close around $x$ at scale $r_\alpha$ to a round shrinking sphere, and by $Z(x)$ the infimum of $r_\alpha$ such that $\mathcal{X}$ is $\eps$-close around $x$ at scale $r_\alpha$ to a round shrinking cylinder or to one of its $\mathbb{Z}_2$-quotients.\\

Now assume $x_i\in \mathcal{X}$ is a sequence of spherical singularities converging to a neck or quotient neck singularity $x\in \mathcal{X}$. Since the flow has a spherical singularity at $x_i$ it clearly holds that $S(x_i)>0$. On the other hand, recall that $x_i\to x$ by definition means convergence in the natural topology \cite[Definition 3.43]{Bam2} and by {\cite[Proposition 3.45]{Bam2}} is equivalent to $\mathfrak{t}(x_i)\to \mathfrak{t}(x)$ and $\nu_{x_i;s}\to \nu_{x;s}$ for all $s<\mathfrak{t}(x)$ in the $W_1$-Wasserstein distance.
{Since the flow has a (quotient) neck singularity at $x$, applying the change of base-point theorem {\cite[Theorem 6.40]{Bam2}}, as recalled in the previous section, we thus obtain a sequence of metric flows around $x_i$ rescaled by  $\lambda_i\to \infty$ which still converges to a (quotient) neck. This implies
\begin{equation}
\lim_{i\to \infty} Z(x_i)=0.
\end{equation}}

\bigskip

Next, recall that the pointed Nash entropy, which is monotone by \cite{Per1}, is defined as
\bea \mathcal{N}_{x_i}(\tau) = \int_{\mathcal{X}_{\mathfrak{t}(x_i)-\tau}} f_{x_i}(y,\mathfrak{t}(x_i)-\tau )\, d\nu_{x_i;\mathfrak{t}(x_i)-\tau }(y) -\frac{n}{2}, \eea
where $f_x$ is the (almost everywhere defined) potential function for the heat kernel measure, namely
\bea d\nu_{x_i;\mathfrak{t}(x_i)-\tau}(y) = (4\pi\tau)^{-3/2} e^{-f_{x_i}(y,\mathfrak{t}(x_i)-\tau)}\, d\mathrm{Vol}_{g_{\mathfrak{t}(x_i)-\tau}}(y) .\eea  
By Hamilton's roundness estimate \cite{Ham} we have $S(x_i)\leq Z(x_i)$. On the other hand, we claim that 
\begin{equation}\label{eq_bdd_ratio}
Z(x_i)\leq C S(x_i).
\end{equation}
To see this, recall that by the almost rigidity case of the monotonicity formula from \cite[Theorem {2.20}]{Bam3} we can find a constant $\delta>0$, such that if $\mathcal{N}_{x_i}(r_{\alpha})-\mathcal{N}_{x_i}(r_{\alpha+\lceil \e^{-1} \rceil})<\delta$ then $\mathcal{X}$ is $\eps$-close around $x_i$ at scale $r_\alpha$ to a metric soliton. Recall that since we are working in dimension 3 the only metric solitons are flat $\mathbb{R}^3$, the round shrinking $S^3$, the round shrinking $\mathbb{R}\times S^2$, as well as finite quotients thereof. Now, if $r_{\alpha_i}=S(x_i)$ is the largest spherical scale, then by monotonicity there must be some $\beta_i \in \{\alpha_{i}+1,\ldots,\alpha_{i}+N\lceil \e ^{-1}   \rceil  \}$, where $N<\infty$ is {a positive integer} independent of $i$, such that $\mathcal{X}$ is $\eps$-close around $x_i$ at scale $r_{\beta_i}$ to a metric soliton. {To see this, note since the flow has a (quotient) neck singularity at $x$, by the semicontinuity of the Nash entropy from \cite[Proposition 4.37]{Bam3}, which is applicable thanks to our assumption that the scalar curvature is bounded below, we get a uniform entropy bound $\mathcal{N}_{x_i}(\tau_0)\geq -Y$ at some fixed scale $\tau_0>0$. Hence, if there was no $\beta_i$ such that $\mathcal{N}_{x_i}(r_{\beta_i})-\mathcal{N}_{x_i}(r_{\beta_i+\lceil \e^{-1} \rceil})<\delta$, then for $i$ large enough by monotonicity we would obtain $\mathcal{N}_{x_i}(r_{\alpha_i+N\lceil \e ^{-1}   \rceil })\leq \mathcal{N}_{x_i}(r_{\alpha_i})-N\delta$, which yields a contradiction provided we choose $N>Y/\delta$.} Now, if $\mathcal{X}$ was $\eps$-close around $x_i$ at scale $r_{\beta_i}$ to flat $\mathbb{R}^3$, then, provided $\eps$ is small enough, the local regularity theorem \cite[Theorem {2.29}]{Bam3} would yield a contradiction with the assumption that $x_i$ is a singular point. Also, if $\mathcal{X}$ was $\eps$-close around $x_i$ at scale $r_{\beta_i}$ to a round shrinking sphere or a round shrinking nontrivial spherical space form, then we would obtain a contradiction with the definition of $S(x_i)$ or with Hamilton's convergence theorem \cite{Ham}, respectively. Hence, $\mathcal{X}$ is $\eps$-close around $x_i$ at scale $r_{\beta_i}$ to a round shrinking $\mathbb{R}\times S^2$ or one of its $\mathbb{Z}_2$-quotients. This shows that $Z(x_i)/S(x_i)\leq 2^{N\lceil \e ^{-1}   \rceil  } $, and thus proves \eqref{eq_bdd_ratio}.\\

Now, considering the rescaled flows $(\mathcal{X}^{x_i,1/S(x_i)}, (\nu^{x_i,1/S(x_i)})_{s\le 0})$ by Bamler's compactness theory \cite{Bam2} we can pass to a subsequential limit $(\mathcal{X}^\infty,(\nu^\infty_{x_{\infty};s})_{s\leq 0})$. By \cite[Theorem {2.44}]{Bam3} (see also \cite{Per1,Per2}) any such 3-dimensional blowup limit $\mathcal{X}^\infty$ is $\kappa$-noncollapsed, has nonnegative curvature, and is smooth -- with bounded curvature on compact time intervals -- until it becomes extinct. Moreover, by construction $\mathcal{X}^\infty$ satisfies $S({x_{\infty}})\geq 1$ and $Z({x_{\infty}})\leq C$. Hence, by the recent classification of compact $\kappa$-solutions from Brendle-Daskalopoulos-Sesum \cite{BDS} we conclude that $\mathcal{X}^\infty$ must be an ancient oval. This proves the proposition.
\end{proof}

\bigskip

\begin{prop}\label{prop2}
If there are only finitely many spherical singularities near $x\in\mathcal{X}$, then Perelman's ancient ovals do not occur as blowup limit at $x$.
\end{prop}

\begin{proof}
Suppose towards a contradiction that there are $x_i \to x$ and $\lambda_i\to\infty$ such that $(\mathcal{X}^{x_i,\lambda_i}, (\nu^{x_i,\lambda_i}_s)_{s\le 0})$  has as $\mathbb{F}$-limit an ancient oval, say $(\mathcal{X}^\infty , (\nu ^\infty_{x_{\max};s})_{s\le 0}) $, that becomes extinct at time zero. Since the time $-1$ slices converge smoothly by the local regularity theorem \cite[Theorem {2.29}]{Bam3}, and since the ancient ovals are compact with positive curvature, it follows that the flow $\mathcal{X}$ has a spherical singularity at some nearby space-time point $y_i \in \mathcal{X}$ with $y_i\to x$. Now, using the change of base-point theorem \cite[Theorem 6.40]{Bam2} we see that $(\mathcal{X}^{y_i,\lambda_i},(\nu^{y_i,\lambda_i}_s)_{s\le 0})$ still $\mathbb{F}$-converges to an ancient oval. However, since there are only finitely many spherical singularities, we infer that $y_i=x$ for large $i$, and passing to a subsequential limit we obtain a tangent flow at $x$, which is selfsimilar. This is a contradiction, and thus proves the proposition.
 \end{proof}

\bigskip 
 
We can thus conclude that our main theorem holds true:
 
\begin{proof}[Proof of Theorem \ref{thm_equivalence}] Fix $x\in \mathcal{X}$. Suppose first that there is an accumulation of spherical singularities at $x$. Note that $x$ must be a singular point by the local regularity theorem \cite[Theorem {2.29}]{Bam3}. {Recall that since we are working in dimension 3, as a consequence of \cite[Theorem {2.44}]{Bam3} and \cite{Per2} the only nontrivial metric solitons are the round shrinking $S^3$ and the round shrinking $\mathbb{R}\times S^2$, as well as finite quotients thereof.} {Note that spherical space forms cannot occur as tangent flow at $x$, since round singularities are isolated. Indeed, as a consequence of Hamilton's classical theorem \cite{Ham}, at any spherical singularity $x$ the tangent flow is unique and becomes extinct in a unique singular point. Hence, if there was a sequence $x_i$ of singular points converging to $x$ in the natural topology, then for $i$ large enough we would have $x_i=x$.} Thus, any tangent flow at $x$ must be a round shrinking cylinder or one of its $\mathbb{Z}_2$-quotients, i.e. the flow $\mathcal{X}$ has a neck or quotient neck singularity at $x$.  Hence, by Proposition \ref{prop1}, Perelman's ancient ovals occur as blowup limit at $x$. Conversely, if there is no accumulation of spherical singularities at $x$, then, by Proposition \ref{prop2}, Perelman's ancient ovals cannot occur as blowup limit at $x$. This proves the theorem.
\end{proof}

\bigskip

\bibliographystyle{amsplain}

\vspace{10mm}

{\sc Beomjun Choi, Department of Mathematics, POSTECH, 77 Cheongam-Ro, Nam-Gu, Pohang, Gyeongbuk, Korea 37673}\\

{\sc Robert Haslhofer, Department of Mathematics, University of Toronto,  40 St George Street, Toronto, ON M5S 2E4, Canada}\\

\end{document}